\documentclass[12pt,reqno]{amsart}
\usepackage{amsmath}
\usepackage{amssymb}
\usepackage{amstext}
\usepackage{mathrsfs}
\usepackage{a4wide}
\usepackage{graphicx}
\usepackage{bm}
\allowdisplaybreaks \numberwithin{equation}{section}
\usepackage{color}
\usepackage{cases}

\usepackage{hyperref}
\hypersetup{hypertex=true,
            colorlinks=true,
            linkcolor=blue,
            anchorcolor=blue,
            citecolor=blue }

\numberwithin{equation}{section}

\newtheorem{theorem}{Theorem}[section]
\newtheorem{proposition}[theorem]{Proposition}

\newtheorem{lemma}[theorem]{Lemma}

\theoremstyle{definition}

\newtheorem{definition}[theorem]{Definition}

\theoremstyle{remark}
\newtheorem{remark}[theorem]{Remark}

\newtheorem{example}[theorem]{Example}

\def\d{\mathrm{d}}

\newcommand{\R}{\mathbb{R}}

\newcommand{\LL}{\mathcal{L}}

\newcommand{\M}{\mathcal{M}}
\newcommand{\T}{\mathcal{T}}
\newcommand{\s}{\mathcal{S}(\R^N)}

\begin{document}

\title[Symmetry in Serrin-type overdetermined problems]{Symmetry in Serrin-type overdetermined problems}

 \author{Daomin Cao, Juncheng Wei, Weicheng Zhan }

\address{State Key Laboratory of Mathematical Sciences, Academy of Mathematics and Systems Science, Chinese Academy of Sciences, Beijing 100190, China}
\email{dmcao@amt.ac.cn}

\address{Department of mathematics, Chinese university of Hong Kong, Shatin, Hong Kong}
\email{jcwei@math.ubc.ca}

\address{School of Mathematical Sciences, Xiamen University, Xiamen 361005, P.R. China}
\email{zhanweicheng@amss.ac.cn}

\begin{abstract}
This paper investigates the geometric constraints imposed on a domain by overdetermined problems for partial differential equations. Serrin’s symmetry results are extended to overdetermined problems with potentially degenerate ellipticity in nonsmooth bounded domains. Furthermore, analogous symmetry results are established for ring-shaped domains. The proof relies on continuous Steiner symmetrization, along with a carefully constructed approximation argument.


\bigskip

\end{abstract}

\maketitle


\bibliographystyle{siam}

\section{Introduction and Main results}

In this paper, we investigate how overdetermined problems for partial differential equations constrain the geometry of the domain, a topic with a rich history spanning more than half a century.

\subsection{Serrin's classical result}
Let us begin by recalling the classical Serrin's overdetermined problem
\begin{align}
    -\Delta u=1,&\ \ \ \text{in}\ \Omega, \label{eq1} \\
    u=0,\ \ \ |\nabla u|=\mathbf{c}, &\ \ \ \text{on}\ \partial \Omega,  \label{eq2}
\end{align}
where $\Omega$ is a bounded domain in $\R^N$ with $N\ge 2$, and $\mathbf{c}$ is a constant. It is easy to see that if $\Omega$ is a ball, then \eqref{eq1}-\eqref{eq2} admits a unique solution, which is radially symmetric. An intriguing inverse problem arises, namely whether the following statement holds true:
\begin{equation}\label{eq0}
  \text{if\ \eqref{eq1}-\eqref{eq2}}\ \text{admits a solution, then}\ \Omega\ \text{is a ball}.
\end{equation}
This problem was first addressed by Serrin in his celebrated paper \cite{Ser1971MR333220} (1971). He proved \eqref{eq0}, assuming that $\partial \Omega$ is of class $C^2$ and $u\in C^2(\overline{\Omega})$. As a direct consequence, $u$ can be explicitly represented as a quadratic function. We would like to mention that this result, in the case $N=2$, admits a nice and insightful interpretation in fluid dynamics. Specifically, $u$ can be thought as the velocity of a homogeneous, incompressible fluid flowing steadily along parallel streamlines within a hollow cylindrical pipe, subject to the no-slip boundary condition, with a constant wall shear stress exerted by the fluid on the pipe wall. In this context, Serrin's result asserts that \emph{the only possible configuration of the pipe is a circular cylinder}. For further physical interpretations of this result, see \cite{Sirakov2002} or \S 8.3 in \cite{PucciMR2356201}.

Serrin's proof is based on the so-called \emph{moving plane method}, originally introduced by Alexandrov \cite{AlexandrovMR102114, AlexandrovMR143162} in the context of geometric problems. Right after Serrin’s paper, Weinberger \cite{WeinbergerMR333221} provided an alternative proof of the same result, employing the maximum principle for an auxiliary function, called $P$-function, along with Poho\v{z}aev identity; this method is nowadays known as the \emph{$P$-function approach}. Following the works of Serrin and Weinberger, several alternative proofs of Serrin's result have appeared in the literature; see, e.g., \cite{BrandoliniMR2448319, Bro2016MR3509374, ChoulliMR1626395, PayneMR1021402}. In addition, a substantial body of literature has extended Serrin’s results in various directions, including systems of quasilinear equations, degenerate elliptic equations, exterior domains, annular domains, weakened conditions, and other boundary conditions; see, e.g., \cite{BrandoliniMR2436453, Bro2016MR3509374, BrockMR1947461, ButtazzoMR2764863, CianchiMR2545870, CiraoloMR3959271, FarinaMR2366129, FarinaMR3145008, Figalli2024, FolMR3086464, FragaMR2863764, FraMR2232009, GarofaloMR980297, KumaresanMR1487977, Lian2025, MagnaniniMR4041100, PrajapatMR1487978, ReichelMR1416582, ReichelMR1463801, RosMR4046014, RosMR3062759, RosMR3666566, RuizMR4575796, ShahgholianMR2916825, SirakovMR1808026, TraizetMR3192039, VogelMR1200301, WangMR3952780} and the references therein. We also refer interested readers to the surveys \cite{NitschMR3802818, SchaeferMR1971630, SicbaldiMR4559541} for a comprehensive overview.

\subsection{Main results}
In this paper, we are interested in the following equation
\begin{equation}\label{eq3}
   -\operatorname {div}\left(g(|\nabla u|)\frac{\nabla u}{|\nabla u|}\right)=f(u),\ \ \ \text{in}\ \Omega.
\end{equation}
Equations of this type arise in various problems from different backgrounds. We refer the reader to \cite{FraMR2232009, PucciMR2356201, Sirakov2002} for a relevant introduction. We impose the following assumptions on $f$ and $g$:
\begin{itemize}
  \item[$(\mathcal{A})\ $] The function $g:[0, +\infty)\to \mathbb{R}$ is continuous and strictly increasing, with $g(0)=0$; The function $f:[0, +\infty)\to \R$ admits the decomposition $f = f_1 + f_2$, where $f_1$ is continuous and $f_2$ is a function of bounded variation.
\end{itemize}
A prototypical example is the capillary overdetermined problem, where $g(z)=\frac{z}{\sqrt{1+z^2}}$ represents the mean curvature operator (see, e.g., \cite{Lian2025}). Another prototypical example is the so-called $p$-Laplace equation, where $g(z)=z^{p-1}$ for $p>1$. This operator is singular for $p<2$ and degenerate for $p>2$. In this case, solutions to the equation are typically only of class $C^{1,\alpha}(\Omega)$; see \cite{DiBenedettoMR709038, TolksdorfMR727034}. It is therefore natural to consider solutions to \eqref{eq3} in a weak sense. We shall say that $u\in C^1(\Omega)$ is a weak solution of \eqref{eq3} if
\begin{equation}\label{eq4}
  \int_\Omega g(|\nabla u|)\frac{\nabla u \cdot \nabla \varphi}{|\nabla u|}\,\d x=\int_{\Omega} f(u)\varphi\,\d x,\ \ \ \forall\,\varphi\in C_0^\infty(\Omega).
\end{equation}
The integrand in \eqref{eq4} is interpreted to be zero at each $x$ where $\nabla u(x)=0$.

The primary goal of this paper is to extend Serrin's result to more general partial differential equations of the form \eqref{eq3} in arbitrary bounded domains.

Since we aim to work in a general bounded domain, the boundary conditions \eqref{eq2} must be understood in an appropriate sense. Following \cite{VogelMR1200301}, we reformulate them as follows:
\begin{equation}\label{eq5}
  u(x)\to 0\ \ \ \text{and}\ \ \ |\nabla u(x)|\to \mathbf{c}\ \ \text{uniformly as}\ \ x\to \partial \Omega,
\end{equation}
namely, for any $\varepsilon>0$, there exists an open set $\mathcal{V}\supset \partial \Omega$ such that
\begin{equation}\label{eq6}
  \left|u(x)\right|<\varepsilon,\ \ \ \left||\nabla u(x)|-\mathbf{c} \right|<\varepsilon,\ \ \ \forall\, x\in \mathcal{V} \cap \Omega.
\end{equation}
Note that if $u\in C^1(\overline{\Omega})$, then \eqref{eq5} is identical to \eqref{eq2}.

Our first main result is the following.

\begin{theorem}\label{th1}
  Let $\Omega \subset \R^N$ be an arbitrary bounded domain. Let $f$ and $g$ be two functions satisfying assumption $(\mathcal{A})$. Suppose that $u\in C^1(\Omega)$ is a positive weak solution of \eqref{eq3} in the sense of \eqref{eq4}, and that $u$ satisfies \eqref{eq5} for some $\mathbf{c}>0$. Then, $\Omega$ must be a ball.
\end{theorem}

We highlight that Theorem \ref{th1} makes no assumptions on the smoothness of $\partial \Omega$. The solution $u$ may not be radially symmetric in general. For example, in the case of the $p$-Laplacian, where $g(z)=z^{p-1}$ for some $p>1$, nonsymmetric solutions to \eqref{eq3} exist in a ball; see \cite{Broc2000MR1758811, BrockMR1947461}. For the reader’s convenience, we recall the example constructed in \cite{Broc2000MR1758811, BrockMR1947461}:
\begin{example}\label{ex1}
  Let $p\ge 2$, $s>2$,
  \begin{equation*}
    w(x)=\begin{cases}
           (1-|x|^2)^s, & \mbox{if }\ |x|\le1, \\
           0, & \mbox{if }\ |x|>1,
         \end{cases}
  \end{equation*}
  and
  \begin{equation*}
    v(x)=\begin{cases}
           1, & \mbox{if }\ |x|<5, \\
           1-\left((|x|^2-25)/11  \right)^s, & \mbox{if }\ 5\le |x|\le 6.
         \end{cases}
  \end{equation*}
  Let $B_r(x)$ denote the open ball in $\mathbb{R}^N$ centered at $x$ with radius $r>0$, and let $B_r:=B_r(0)$. We choose $x^1, x^2\in B_4$ with $|x^1-x^2|>2$ and set
  \begin{equation*}
    u(x):=v(x)+w(x-x^1)+w(x-x^2),\ \ \ x\in B_6.
  \end{equation*}
  The graph of $u$ consists of three radially symmetric ``mountains", one of which has a ``plateau" at height 1, while the other two are congruent and positioned so that their ``bases" rest on the plateau; see Figure 1 in \cite{Broc2000MR1758811}. After a short computation we see that $u$ satisfies
  \begin{align*}
    -\Delta_p u\equiv -\operatorname {div}\left(|\nabla u|^{p-2} \nabla u\right)=f(u),\ \ \ u>0,&\ \ \ \text{in}\ B_6,  \\
    u=0,\ \ \ |\nabla u|=\frac{12s}{11}, &\ \ \ \text{on}\ \partial B_6,
\end{align*}
where
\begin{equation*}
  f(u):=\begin{cases}
          (2s/11)^{p-1}\left(25+11(1-u)^{1/s} \right)^{(p/2)-1}\left(1-u\right)^{p-(p/s)-1}\times\ \  &  \\
          \times \left\{(50/11)(p-1)(s-1)+(2ps-2s-p+n)(1-u)^{1/s} \right\} , & \mbox{if }\ 0\le u\le 1, \\
          (2s)^{p-1}\left(1-(u-1)^{1/s}\right)^{(p/2)-1}(u-1)^{p-(p/s)-1}\times\ \ \  & \\
          \times \left\{-2(s-1)(p-1)+(2ps-2s-p+n)(u-1)^{1/s} \right\},\ \ \  & \mbox{if }\ 1\le u\le 2.
        \end{cases}
\end{equation*}
If $p=2$ and $s>2$, then we have $f\in C^{\infty}\left([0, 2]\backslash\{1\} \right)\cap C^{1-(2/s)}\left([0, 2]\right)$. If $p>2$ and $s>p/(p-2)$, then we have $f\in C^1\left([0, 2]\right)$.
\end{example}

Naturally, under suitable additional conditions, it can be further established that $u$ is radially symmetric; see, e.g., \cite{PucciMR2356201, Sirakov2002, VogelMR1200301}. Observe also that the constant $\mathbf{c}$ is assumed to be positive. This condition is imposed mainly to prevent $u$ from degenerating near the boundary. It also plays a crucial role in \cite{VogelMR1200301}, where it is used to show that $u$ is comparable to the distance to $\partial \Omega$. As noted in \cite{FolMR3086464}, the degenerate case $\mathbf{c}=0$ arises in various contexts and is therefore of considerable importance. Our second main result specifically addresses this case. In fact, similar results can still be established for $\mathbf{c}=0$ under assumptions slightly stronger than those in \eqref{eq6}. Specifically, we require that for any $\varepsilon>0$, there exists an open set $\mathcal{V}\supset \partial \Omega$ such that
\begin{equation}\label{eq7}
  \left|u(x)\right|<\varepsilon,\ \ \ 0<|\nabla u(x)|<\varepsilon,\ \ \ \forall\, x\in \mathcal{V} \cap \Omega.
\end{equation}
Note that in \eqref{eq7}, the condition requiring $|\nabla u|$ to remain nonvanishing near the boundary $\partial \Omega$ is imposed to prevent potential severe degeneracy in this region.

Our second main result is as follows, serving as a complement to Theorem \ref{th1}.
\begin{theorem}\label{th2}
  Let $\Omega \subset \R^N$ be an arbitrary bounded domain. Let $f$ and $g$ be two functions satisfying assumption $(\mathcal{A})$. Suppose that $u\in C^1(\Omega)$ is a positive weak solution of \eqref{eq3} in the sense of \eqref{eq4}, and that $u$ satisfies \eqref{eq7}. Then, $\Omega$ must be a ball.
\end{theorem}


Let us now briefly review some relevant works. In 1989, Garofalo and Lewis \cite{GarofaloMR980297} successfully extended Weinberger's argument (i.e., the $P$-function approach) to more general partial differential equations of the form \eqref{eq3} in a general bounded domain, assuming $f\equiv 1$ and imposing additional growth conditions on $g$. Under the same growth conditions on $g$, Brock and Henrot \cite{BrockMR1947461} established similar symmetry results using continuous Steiner symmetrization and the domain derivative, initially assuming that $\Omega$ is convex. Fragal\`a, Gazzola, and Kawohl \cite{FraMR2232009} removed these growth conditions and provided a simpler, more geometric proof. However, the proof in \cite{FraMR2232009} required $\partial \Omega\in C^{2, \alpha}$ and an additional starshapedness assumption on $\Omega$ when the dimension $N$ of $\Omega$ is greater than $2$. The starshapedness assumption was later removed by Farina and Kawohl in \cite{FarinaMR2366129} using a modified $P$-function approach. We would like to point out the $P$-function approach essentially relies on the fact that the right-hand side of \eqref{eq3} is a constant (i.e., $f\equiv \text{const.}$); see \cite{GazzolaMR2257026} for further discussion on this point. A totally different line of reasoning was pursued in \cite{Bro2016MR3509374}, where Brock established an analogue of Theorem \ref{th1} using continuous Steiner symmetrization (see \cite{Bro1995MR1330619, Broc2000MR1758811}). Further efforts have been made to extend Serrin's original results in \cite{Ser1971MR333220} to nonsmooth domains; see, e.g., \cite{CortMR1387457, Figalli2024, FolMR3086464, PrajapatMR1487978, VogelMR1200301}. We recommend that the reader refer to the recent works \cite{Figalli2024, FolMR3086464} for a comprehensive discussion on this topic. Remarkably, combining Weinberger’s method with tools from geometric measure theory, Figalli and Zhang \cite{Figalli2024} succeeded in establishing \eqref{eq0}, where \eqref{eq1}-\eqref{eq2} are both understood in the weak sense on domains of finite perimeter with a uniform upper bound on the density. However, their approach appears to rely heavily on the specific form of \eqref{eq1}, and thus does not seem applicable to the potentially degenerate elliptic equation \eqref{eq3}. It would be of interest to determine whether the boundary conditions considered here can be relaxed to match those in \cite{Figalli2024}.

The second goal of this paper is to further extend the above results to the so-called ring-shaped domain.
\begin{definition}[Ring-shaped domain]\label{def1}
  Let $\Omega_0, \Omega_1$ be two bounded domains in $\mathbb{R}^N$ such that $\overline{\Omega_1}\subset \Omega_0$. If $\Omega:=\Omega_0\backslash \overline{\Omega_1}$ is connected and satisfies $\partial \Omega=\partial \Omega_0\cup \partial \Omega_1$, then it is called a \emph{ring-shaped domain}.
\end{definition}
\noindent For a ring-shaped domain $\Omega=\Omega_0\backslash \overline{\Omega_1}$, the overdetermined conditions are typically given as follows:
\begin{align}
    u=\eta,\ \ \ \partial_{\nu}u=\mathbf{c}_1,&\ \ \ \text{on}\ \partial\Omega_1, \label{eq11} \\
    u=0,\ \ \ \partial_{\nu}u=\mathbf{c}_0, &\ \ \ \text{on}\ \partial \Omega_0,  \label{eq12}
\end{align}
where $\nu$ denotes the \emph{inner} unit normal to $\Omega$, and $\eta$, $\mathbf{c}_0$, and $\mathbf{c}_1$ are constants. In recent years, there has been tremendous interest in the study of overdetermined problems in ring-shaped domains; see, e.g., \cite{Agostiniani2024, BorghiniMR4431667, CavallinaMR4340795, Enci2023, KamburovMR4200475, KhavinsonMR2174103, LeeMR4566201, ReichelMR1416582, RuizMR4575796, SirakovMR1808026, WillmsMR1289661} and the references therein. In \cite{ReichelMR1416582}, Reichel studied the semilinear elliptic equation
\begin{equation}\label{eq13}
  -\Delta u=f(u),
\end{equation}
subject to the overdetermined conditions \eqref{eq11}-\eqref{eq12}. Under the additional assumption that $0<u<\eta$ in $\Omega$, he proved that $\Omega$ must be a standard annulus and that $u$ is radially symmetric. Several years later, Sirakov \cite{SirakovMR1808026} successfully relaxed this assumption and further proved that the conclusion remains valid even when allowing different values of $\eta$ and $\mathbf{c}_1$ on different connected components of $\partial \Omega_1$. Recently, Ruiz \cite{RuizMR4575796} extended Reichel's results by allowing the nonlinearity $f$ in \eqref{eq13} to be merely continuous at the endpoints and, under additional assumptions, covering the degenerate case where $\mathbf{c}_i$ vanishes, $i=0, 1$  (see also \cite{Drivas2024, Wang2023} for related results and applications).

By employing bifurcation theory, researchers have constructed various exceptional ring-shaped domains arising in the study of overdetermined problems. In \cite{KamburovMR4200475}, Kamburov and Sciaraffia constructed a bounded real-analytic ring-shaped domain $\Omega$, distinct from a standard annulus, in which the overdetermined problem \eqref{eq1}, \eqref{eq11}, and \eqref{eq12} admits a solution $u\in C^\infty(\overline{\Omega})$ with $\eta>0$ and $\mathbf{c}_1=\mathbf{c}_0>0$. Recently, Agostiniani, Borghini and Mazzieri \cite{Agostiniani2024} demonstrated the existence of infinitely many planar ring-shaped domains $\Omega$, also distinct from a standard annulus, where the same overdetermined problem admits a solution $u\in C^\infty(\overline{\Omega})$ with $\eta=0$. Furthermore, they proved that if $u$ has infinitely many maximum points, then $\Omega$ must be a standard annulus and that $u$ is radially symmetric. Very recently, Enciso, Fernández, Ruiz, and Sicbaldi \cite{Enci2023} considered equation \eqref{eq13} with $f(z)=\lambda z$ for some $\lambda\in \R$, subject to the overdetermined conditions \eqref{eq11}–\eqref{eq12}, with $\eta>0$ and $\mathbf{c}_1=\mathbf{c}_0=0$. It is worth mentioning that this overdetermined problem is closely related to the well-known Schiffer conjecture; see \cite{KawohlMR4205793} for further discussion on this conjecture. They successfully constructed a family of non-symmetric planar ring-shaped domains $\Omega$ for which the overdetermined problem admit a nontrivial solution.

To the best of our knowledge, no established results are currently available on overdetermined problems involving potentially degenerate ellipticity in ring-shaped domains. Our next objective is to extend Theorems \ref{th1} and \ref{th2} to ring-shaped domains, under the assumption that $0<u<\eta$ in $\Omega$, as in \cite{ReichelMR1416582}. Note that consistency requires $\mathbf{c}_1\le 0\le \mathbf{c}_0$. Analogous to \eqref{eq6} and \eqref{eq7}, we introduce the following weak form of the overdetermined conditions \eqref{eq11}-\eqref{eq12}: For any $\varepsilon>0$, there exist open sets $\mathcal{V}_1\supset \partial \Omega_1$ and $\mathcal{V}_0\supset \partial \Omega_0$, such that
\begin{equation}\label{eq14}
\begin{split}
     &  \left|u(x)-\eta\right|<\varepsilon\ \ \text{and}\ \ \left||\nabla u(x)|+\mathbf{c}_1 \right|<\varepsilon\ \text{if}\ \mathbf{c}_1\not=0;\ 0<|\nabla u(x)|<\varepsilon\ \text{if}\ \mathbf{c}_1=0,\ \ \forall\, x\in \mathcal{V}_1 \cap \Omega; \\
     & \left|u(x)\right|<\varepsilon\ \ \text{and}\ \ \left||\nabla u(x)|-\mathbf{c}_0 \right|<\varepsilon\ \text{if}\ \mathbf{c}_0\not=0;\ 0<|\nabla u(x)|<\varepsilon\ \text{if}\ \mathbf{c}_0=0,\ \ \forall\, x\in \mathcal{V}_0 \cap \Omega. \\
\end{split}
\end{equation}
Clearly, if $\partial \Omega$ is smooth, $u\in C^1(\overline{\Omega})$, and $0\le u \le \eta$ in $\Omega$, then \eqref{eq14} coincides with \eqref{eq11}-\eqref{eq12}.

Our last result is stated as follows.
\begin{theorem}\label{th3}
  Let $\Omega=\Omega_0\backslash \overline{\Omega_1}$ be an arbitrary ring-shaped domain as defined in Definition \ref{def1}. Let $f$ and $g$ be two functions satisfying assumption $(\mathcal{A})$. Let $\eta>0$ and $\mathbf{c}_1\le 0\le \mathbf{c}_0$. Suppose $u\in C^1(\Omega)$ is a weak solution of \eqref{eq3} in the sense of \eqref{eq4}, and that $u$ satisfies $0<u<\eta$ in $\Omega$ and \eqref{eq14}. Then, $\Omega_0$ and $\Omega_1$ must be balls.
\end{theorem}

It is important to note that the above conclusion does not assert that $\Omega_1$ and $\Omega_0$ are concentric, i.e., that $\Omega$ is a standard annulus. In fact, as the following example illustrates, $\Omega$ may be \emph{nonconcentric}, showing that the conclusion is already optimal.

\begin{example}
  Let $w$ and $v$ be as in Example \ref{ex1}. Let $x^1$ be an arbitrary point in $B_2$. Let
  \begin{equation*}
    u(x)=v(x)+w(x-x^1),\ \ \ x\in B_6.
  \end{equation*}
  We check that $u$ satisfies
  \begin{align*}
    -\Delta_p u\equiv -\operatorname {div}\left(|\nabla u|^{p-2} \nabla u\right)=f(u),\ \ \ u>0,&\ \ \ \text{in}\ B_6\backslash \overline{B_{\frac{1}{2}}(x^1)},  \\
    u=(3/4)^s,\ \ \ |\nabla u|=s(3/4)^{s-1}, &\ \ \ \text{on}\ \partial B_{\frac{1}{2}}(x^1), \\
     u=0,\ \ \ |\nabla u|=\frac{12s}{11}, &\ \ \ \text{on}\ \partial B_6,
\end{align*}
where $f$ is given as in Example \ref{ex1}.
\end{example}

We note that if $\Omega_i$ (for $i = 1$ or $2$) is known \emph{a priori} to be a ball, then the Neumann boundary condition for $u$ on $\partial\Omega_i$ becomes superfluous and may be omitted; see Remark \ref{remark-4} in Section \ref{s4} for further details. Similar results have also been discussed in \cite{ReichelMR1416582, SirakovMR1808026}.

\subsection{Idea of the proof}
The central idea of the proof is that it suffices to establish the local symmetry of $u$, as introduced by Brock \cite{Broc2000MR1758811} (see Section \ref{s2} below for the precise definition). Local symmetry is a concept weaker than global symmetry, yet it is sufficient for our purposes. To achieve this, we employ the continuous Steiner symmetrization method, developed by Brock \cite{Bro1995MR1330619, Broc2000MR1758811, Bro2016MR3509374}. A major difficulty arises from the fact that the domain is a priori unknown, necessitating a suitable truncation and approximation procedure. In \cite{Bro2016MR3509374}, Brock proposed an approximation method based on a carefully designed truncation near the boundary. However, his argument appears to rely crucially on the assumption $\mathbf{c}>0$ and does not directly extend to the degenerate case $\mathbf{c}=0$. Inspired by \cite{Bro2016MR3509374}, we develop an alternative approximation scheme that appears to be more concise and effective. The key insight is to exploit a certain monotonicity property; see Lemmas \ref{le1} and \ref{le6} for details. In addition, several technical estimates are required to complete the argument.
\subsection{Organization of the paper}
The rest of the paper is organized as follows: In Section \ref{s2}, we present some preliminary results that will be used in the proofs. Section \ref{s3} is devoted to proving Theorems \ref{th1} and \ref{th2}. In Section \ref{s4}, we provide the proof of Theorem \ref{th3}.

\section{The continuous Steiner symmetrization}\label{s2}

Our main tool in the proof is the so-called continuous Steiner symmetrization method, developed by Brock \cite{Bro1995MR1330619, Broc2000MR1758811, Bro2016MR3509374} (see the survey \cite{Bro2007MR2569330} for a comprehensive introduction to this topic). For reader's convenience, we provide a concise overview of this method in the present section. We remark that the form presented below is the one we will use; the original result holds in more general settings.

\subsection{The continuous Steiner symmetrization}
We will follow the presentation in \cite{Bro2016MR3509374}. Let us start with some notation. Let $\mathcal{L}^N$ denote $N$-dimensional Lebesgue measure. By $\mathcal{M}(\R^N)$ we denote the family of Lebesgue measurable sets in $\R^N$ with finite measure. For a function $u:\R^N\to \R$, let $\{u>a\}$ and $\{b\ge u>a\}$ denote the sets $\left\{x\in \R^N: u(x)>a \right\}$ and $\left\{x\in \R^N: b\ge u(x)>a \right\}$, respectively, ($a, b\in \R$, $a<b$). Let $\mathcal{S}(\R^N)$ be the set of real-valued, nonnegative measurable functions $u$ that satisfy
\begin{equation*}
  \LL^N(\{u>c\})<+\infty,\ \ \forall\, c>0.
\end{equation*}
We first recall the definition of classical Steiner symmetrization; see, for example, \cite{Bro2007MR2569330, Kaw1985MR810619, Lie2001MR1817225, TalentiMR3503198}.
\begin{definition}[Steiner symmetrization]
\ \ \
  \begin{itemize}
    \item [(i)]For any set $M\in \M(\R)$ let
    \begin{equation*}
      M^*:=\left(-\frac{1}{2}\LL^1(M),\  \frac{1}{2}\LL^1(M)\right).
    \end{equation*}
    \item [(ii)]Let $M\in \M(\R^N)$. For every $x'\in \R^{N-1}$ let
    \begin{equation*}
      M(x'):=\left\{x_1\in \R: (x_1, x')\in M \right\}.
    \end{equation*}
    The set
    \begin{equation*}
      M^*:=\left\{ x=(x_1, x'): x_1\in \left(M(x') \right)^*, x'\in \R^{N-1}\right\}.
    \end{equation*}
    is called the Steiner symmetrization of $M$ (with respect to $x_1$).
    \item [(iii)]If $u\in \s$, then the function
    \begin{equation*}
      u^*(x):=\begin{cases}
                \sup \left\{c>0: x\in \left\{u>c \right\}^*\right\}, & \mbox{if }\  x\in \bigcup_{c>0} \left\{u>c \right\}^*, \\
                0, & \mbox{if }\  x\not\in \bigcup_{c>0} \left\{u>c \right\}^*,
              \end{cases}
    \end{equation*}
   is called the Steiner symmetrization of $u$ (with respect to $x_1$).
  \end{itemize}
\end{definition}

\begin{definition}[Continuous symmetrization of sets in $\M(\R)$]
  A family of set transformations
  \begin{equation*}
    \T_t:\  \M(\R)\to \M(\R),\ \ \  0\le t\le +\infty,
  \end{equation*}
is called a continuous symmetrization on $\R$ if it satisfies the following properties: ($M, E\in \M(\R)$, $0\le s, t\le +\infty$)
\begin{itemize}
  \item [(i)]Equimeasurability property:\, $\LL^1(\T_t(M))=\LL^1(M)$,

    \smallskip
  \item [(ii)]Monotonicity property:\, If $M\subset E$, then $\T_t(M)\subset \T_t(E)$,

    \smallskip
  \item [(iii)]Semigroup property:\, $\T_t(\T_s(M))=\T_{s+t}(M)$,

    \smallskip
  \item [(iv)]Interval property:\, If $M$ is an interval $[x-R,\ x+R]$, ($x\in \R$, $R>0$), then $\T_t(M):=[xe^{-t}-R,\ xe^{-t}+R]$,

    \smallskip
  \item [(v)]Open/compact set property: If $M$ is open/compact, then $\T_t(M)$ is open/compact.
\end{itemize}
\end{definition}

For the construction of the family $\T_t$, $0 \le t \le +\infty$, we refer the reader to \cite[Theorem 2.1]{Broc2000MR1758811}.

\begin{definition}[Continuous Steiner symmetrization (CStS)]\label{csts}
  \ \ \
  \begin{itemize}
    \item [(i)]Let $M\in \M(\R^N)$. The family of sets
    \begin{equation*}
      \T_t(M):=\left\{x=(x_1, x'): x_1\in \T_t(M(x')), x'\in \R^{N-1} \right\},\ \ \ 0\le t\le +\infty,
    \end{equation*}
    is called the continuous Steiner symmetrization (CStS) of $M$ (with respect to $x_1$).
    \item [(ii)]Let $u\in \s$. The family of functions $\T_t(u)$, $0\le t \le +\infty$, defined by
    \begin{equation*}
      \T_t(u)(x):=\begin{cases}
                \sup \left\{c>0: x\in \T_t\left(\left\{u>c \right\}\right)\right\}, & \mbox{if }\ x\in \bigcup_{c>0} \T_t\left(\left\{u>c \right\}\right), \\
                0, & \mbox{if }\ x\not\in \bigcup_{c>0} \T_t\left(\left\{u>c \right\}\right),
              \end{cases}
    \end{equation*}
    is called CStS of $u$ (with respect to $x_1$).
  \end{itemize}
\end{definition}
For convenience, we will henceforth denote $M^t$ and $u^t$ as $\T_t(M)$ and $\T_t(u)$, respectively, for $t \in [0, +\infty]$. We summarize below some basic properties of CStS, established by Brock in \cite{Bro1995MR1330619, Broc2000MR1758811}.

\begin{proposition}\label{pro0}
  Let $M\in \M(\R^N)$, $u,v\in \s,\, t\in [0,+\infty]$. Then
\begin{itemize}
    \item [(1)]Equimeasurability:
    \begin{equation*}
      \LL^N(M)=\LL^N(M^t)\ \ \ \text{and}\ \ \ \left\{u^t>c \right\}=\left\{u>c \right\}^t,\  \forall\, c>0.
    \end{equation*}

    \smallskip
    \item [(2)]Monotonicity: If $u\le v$, then $u^t\le v^t$.

    \smallskip
    \item [(3)]Commutativity: If $\phi: [0, +\infty)\to [0, +\infty)$ is bounded and nondecreasing with $\phi(0)=0$, then
    \begin{equation*}
      \phi(u^t)=[\phi(u)]^t.
    \end{equation*}

    \smallskip
    \item [(4)]Homotopy:
    \begin{equation*}
      M^0=M,\ \ \  u^0=u,\ \ \ M^\infty =M^*,\ \ \ u^\infty=u^*.
    \end{equation*}
Furthermore, from the construction of the CStS it follows that, if $M=M^*$ or $u=u^*$, then $M^t=M$, respectively, $u=u^t$ for all $t\in [0, +\infty]$.

    \smallskip
        \item [(5)]Cavalieri's pinciple: If $F$ is continuous and if $F(u)\in L^1(\R^N)$ then
        \begin{equation*}
          \int_{\R^N} F(u)\d x= \int_{\R^N} F(u^t)\d x.
        \end{equation*}

    \smallskip
    \item [(6)]Continuity in $L^p$: If $t_n\to t $ as $n\to +\infty$ and $u\in L^p(\R^N)$ for some $p\in [1, +\infty)$, then
    \begin{equation*}
      \lim_{n\to +\infty}\|u^{t_n}-u^t\|_p=0.
    \end{equation*}

    \smallskip
        \item [(7)]Nonexpansivity in $L^p$: If $u, v\in L^p(\R^N)$ for some $p\in [1, +\infty)$, then
        \begin{equation*}
          \|u^{t}-v^t\|_p\le \|u-v\|_p.
        \end{equation*}

          \smallskip
        \item [(8)]Hardy-Littlewood inequality: If $u, v\in L^2(\R^N)$ then
        \begin{equation*}
           \int_{\R^N} u^t v^t\d x\ge \int_{\R^N} u v\d x.
        \end{equation*}

         \smallskip
        \item [(9)]If $u$ is Lipschitz continuous with Lipschitz constant $L$, then $u^t$ is Lipschitz continuous, too, with Lipschitz constant less than or equal to $L$.

         \smallskip
        \item [(10)] If $\text{supp}\, u\subset B_R$ for some $R>0$, then we also have $\text{supp}\, u^t\subset B_R$. If, in addition, $u$ is Lipschitz continuous with Lipschitz constant $L$, then we have
            \begin{equation*}
              |u^t(x)-u(x)|\le LR\, t,\ \ \ \forall\, x\in B_R.
            \end{equation*}
           Furthermore, there holds
           \begin{equation}\label{6-1}
             \int_{B_R}G(|\nabla u^t|)\d x \le  \int_{B_R}G(|\nabla u|)\d x,
           \end{equation}
           for every convex function $G: [0, +\infty) \to [0, +\infty)$ with $G(0)=0$.
  \end{itemize}
\end{proposition}

\subsection{Local symmmetry}
Following Brock \cite{Broc2000MR1758811}, we introduce a local version of symmetry for a function $u\in \s$.
\begin{definition}[Local symmetry in a certain direction]\label{def2-1}
Let $u\in \s$ be continuously differentiable on $\{x\in \R^N: 0<u(x)<\sup_{\R^N}u \}$, and suppose that this last set is open. Further, suppose that $u$ has the following property. If $y=(y_1, y')\in \R^N$ with
  \begin{equation*}
    0<u(y)<\sup_{\R^N} u,\ \ \ \partial_1 u(y)>0,
  \end{equation*}
  and $\tilde{ y}$ is the (unique) point satisfying
  \begin{equation*}
    \tilde{y}=(\tilde{y}_1, y'),\ \ \ \tilde{y}_1>y_1,\ \ \ u(y)=u(\tilde{y})< u(s, y'),\ \ \forall\, s\in (y_1,\tilde{y}_1),
  \end{equation*}
  then
  \begin{equation*}
  \begin{split}
     \partial_i u(y) & =\partial_iu(\tilde{y}),\ \ \ i=2, \cdots, N, \\
      \partial_1u(y) & =-\partial_1 u (\tilde{y}).
  \end{split}
  \end{equation*}
  Then $u$ is called \emph{locally symmetric in the direction $x_1$}.
\end{definition}

Suppose that for arbitrary rotations $x\mapsto y=(y_1, y')$ of the coordinate system, $u$ is locally symmetric in the direction $y_1$. Then $u$ is said to be \emph{locally symmetric}. In other words, a function $u$ is said to be locally symmetric if it is locally symmetric in \emph{every} direction.

Although locally symmetric functions are not globally radial, they possess strong symmetric characteristics. Roughly speaking, it is radially symmetric and radially decreasing in some annuli (probably infinitely many) and flat elsewhere. For convenience, we denote by $Q_r(x)$ the closed ball in $\mathbb{R}^N$ centered at $x$ with radius $r\ge 0$, and let $Q_r:=Q_r(0)$.
\begin{proposition}[ \cite{Broc2000MR1758811}, Theorem 6.1]\label{pro1}
  Let $u\in \s$ be a locally symmetric function. Set $V:= \{x\in \R^N: 0<u(x)<\sup_{\R^N}u \}$. Then, we have the following decomposition:
    \begin{itemize}
    \item [(1)]$\displaystyle V=\bigcup_{k\in K} A_k \cup \{x\in V: \nabla u(x)=0\}$, \text{where}
    \begin{equation*}
      A_k=B_{R_k}(z_k)\backslash {Q_{r_k}(z_k)},\ \ \ z_k\in \R^N,\ \ \ 0\le r_k<R_k;
    \end{equation*}
     \item[(2)]$K$ is a countable set;
        \smallskip
    \item [(3)] the sets $A_k$ are pairwise disjoint;
    \smallskip
     \item [(4)]$u(x)=U_k(|x-z_k|)$, $x\in A_k$, where $U_k\in C^1([r_k, R_k])$;
         \smallskip
    \item [(5)]$U'_k(r)<0$ for $r\in (r_k, R_k)$;
        \smallskip
    \item [(6)]$u(x)\ge U_k(r_k),\ \forall\,x\in Q_{r_k}(z_k)$, $k\in K$.
  \end{itemize}
\end{proposition}
It can be seen that if $u\in \s$ is locally symmetric, then the super-level sets $\{u>t\}$ $(t\ge 0)$ are countable unions of mutually disjoint balls, and $|\nabla u|=\text{const.}$ on the boundary of each of these balls.

\subsection{A symmetry criterion due to F. Brock}

The following symmetry criterion is due to Brock \cite{Broc2000MR1758811}.
\begin{proposition}[\cite{Broc2000MR1758811}, Theorem 6.2]\label{pro2}
Let $u\in H^1(\R^N)\cap C(\R^N)$ be a nonnegative function with compact support. Recalling Definition \ref{csts}, let $u^t$ denote the CStS of $u$ with respect to $x_1$. Suppose that $u$ is continuously differentiable on $V$, where
\begin{equation*}
  V=\left\{x\in \R^N: 0<u(x)<\sup_{\R^N}u\right\}.
\end{equation*}
Let $G:[0,+\infty) \to [0, +\infty)$ be strictly convex with $G(0)=0$. Then, if
  \begin{equation*}
    \lim_{t\to 0}\frac{1}{t}\left(\int_{\R^N} G\left(|\nabla u^t|\right)\,\d x-\int_{\R^N} G\left(|\nabla u|\right)\,\d x\right)=0,
  \end{equation*}
  $u$ is locally symmetric in the direction $x_1$.
\end{proposition}

\section{Proofs of Theorems \ref{th1} and \ref{th2}}\label{s3}

In this section, we present the proofs of Theorems \ref{th1} and \ref{th2}. Let $u$ be as given in Theorem \ref{th1} or Theorem \ref{th2}. For convenience, we extend $u$ by zero to a function in $\R^N$, retaining the notation $u$. It follows from \eqref{eq6} or \eqref{eq7} that $u$ is Lipschitz in $\R^N$.

Our goal is to prove that $\Omega$ is a ball. Observe that if $u$ is locally symmetric, then the super-level set $\Omega=\{u>0\}$ consists of a countable union of mutually disjoint balls. Since $\Omega$ is connected, it must be a single ball. Thus, our task now is reduced to proving that $u$ is locally symmetric. To this end, we begin by establishing  several technical lemmas.

\subsection{Several technical lemmas}
Recalling Definition \ref{csts}, let $u^t$ denote the CStS of $u$ with respect to $x_1$. Let
\begin{equation*}
  G(z):=\int_{0}^{z}g(s)\,\d s.
\end{equation*}
Clearly, $G:[0,+\infty)\to [0,+\infty)$ is a strictly convex function satisfying $G(0)=0$.
\begin{lemma}[Monotonicity lemma]\label{le1}
  Let $\gamma_0$ and $\gamma_1$ be two positive constants satisfying $\gamma_0\ge \gamma_1$. Then
\begin{equation*}
\begin{split}
     \int_{\R^N}G\left(|\nabla (u-\gamma_0)_+^t|\right)\d x& -\int_{\R^N}G\left(|\nabla (u-\gamma_0)_+|\right)\d x \\
     & \ge \int_{\R^N}G\left(|\nabla (u-\gamma_1)_+^t|\right)\d x-\int_{\R^N}G\left(|\nabla (u-\gamma_1)_+|\right)\d x
\end{split}
\end{equation*}
for all $t\ge0$.
\end{lemma}

\begin{proof}
  Note that $u$ is bounded in $\R^N$. In view of Proposition \ref{pro0} (3), we observe that $(u-\gamma_0)_+^t=(u^t-\gamma_0)_+$. Thus, we have
  \begin{equation*}
  \begin{split}
       \int_{\R^N}G\left(|\nabla (u-\gamma_0)_+^t|\right)\d x  &-\int_{\R^N}G\left(|\nabla (u-\gamma_0)_+|\right)\d x \\
       &  = \int_{\{u^t>\gamma_0\}}G\left(|\nabla u^t|\right)\d x -\int_{\{u^t>\gamma_0\}}G\left(|\nabla u|\right)\d x.
  \end{split}
  \end{equation*}
  Similarly,
  \begin{equation*}
  \begin{split}
       \int_{\R^N}G\left(|\nabla (u-\gamma_1)_+^t|\right)\d x  &-\int_{\R^N}G\left(|\nabla (u-\gamma_1)_+|\right)\d x \\
       &  = \int_{\{u^t>\gamma_1\}}G\left(|\nabla u^t|\right)\d x -\int_{\{u^t>\gamma_1\}}G\left(|\nabla u|\right)\d x.
  \end{split}
  \end{equation*}
  We now estimate
  \begin{equation*}
  \begin{split}
          \int_{\R^N}&G\left(|\nabla (u-\gamma_0)_+^t|\right)\d x -\int_{\R^N}G\left(|\nabla (u-\gamma_0)_+|\right)\d x \\
       &   = \int_{\{u^t>\gamma_0\}}G\left(|\nabla u^t|\right)\d x -\int_{\{u>\gamma_0\}}G\left(|\nabla u|\right)\d x \\
       & = \left(\int_{\{u^t>\gamma_1\}}G\left(|\nabla u^t|\right)\d x-\int_{\{u^t>\gamma_1\}}G\left(|\nabla u|\right)\d x\right)\\
       &\ \ \ \ \ \ \ -\left( \int_{\{\gamma_0\ge u^t>\gamma_1\}}G\left(|\nabla u^t|\right)\d x-\int_{\{\gamma_0\ge u^t>\gamma_1\}}G\left(|\nabla u|\right)\d x   \right)\\
       & = \left(\int_{\R^N}G\left(|\nabla (u-\gamma_1)_+^t|\right)\d x-\int_{\R^N}G\left(|\nabla (u-\gamma_1)_+|\right)\d x\right)\\
       &\ \ \ \ \ \ \ -\left( \int_{\{\gamma_0\ge u^t>\gamma_1\}}G\left(|\nabla u^t|\right)\d x-\int_{\{\gamma_0\ge u^t>\gamma_1\}}G\left(|\nabla u|\right)\d x   \right).
  \end{split}
  \end{equation*}
  So it remains to show that
  \begin{equation}\label{3-2}
   \int_{\{\gamma_0\ge u^t>\gamma_1\}}G\left(|\nabla u^t|\right)\d x-\int_{\{\gamma_0\ge u>\gamma_1\}}G\left(|\nabla u|\right)\d x\le 0.
  \end{equation}
  Consider a nondecreasing function $\phi: \R \to \R$ defined by
  \begin{equation*}
    \phi(s)=\begin{cases}
              \gamma_0-\gamma_1, & \mbox{if}\  s\ge \gamma_0, \\
              s-\gamma_1, & \mbox{if}\  \gamma_1<s<\gamma_0, \\
              0, & \mbox{if}\  s\le \gamma_1.
            \end{cases}
  \end{equation*}
By the definition of $\phi$, we first have
  \begin{equation}\label{3-3}
     \int_{\R^N} G\left(|\nabla \phi(u)|\right)\d x = \int_{\{\gamma_0\ge u>\gamma_1\}}G\left(|\nabla u|\right)\d x.
  \end{equation}
    By Proposition \ref{pro0} (3), we observe that $[\phi(u)]^t=\phi(u^t)$. Hence, it follows that
\begin{equation}\label{3-4}
  \int_{\R^N} G\left(|\nabla [\phi(u)]^t|\right)\d x  = \int_{\R^N} G\left(|\nabla \phi(u^t)|\right)\d x=\int_{\{\gamma_0\ge u^t>\gamma_1\}}G\left(|\nabla u^t|\right)\d x.
\end{equation}
On the other hand, Proposition \ref{pro0} (10) implies that
  \begin{equation}\label{3-5}
    \int_{\R^N}G\left( |\nabla [\phi(u)]^t|\right)\d x\le \int_{\R^N} G\left(|\nabla \phi(u)|\right)\d x.
  \end{equation}
Now, \eqref{3-2} follows directly from the combination of \eqref{3-3}, \eqref{3-4}, and \eqref{3-5}. The proof is thus complete.
\end{proof}

Set $h(z):=zg(z)-G(z)$, $z\in[0, +\infty)$. It is readily verified that $h$ is a strictly increasing function.
\begin{lemma}\label{le2}
  Let $\gamma>0$ be a fixed constant. Then
  \begin{equation}\label{3-7}
  \begin{split}
      &   \int_{\{u^t>\gamma\}}G(|\nabla u^t|)\d x-\int_{\{u>\gamma\}}G(|\nabla u|)\d x\ge \\
       &  \int_{\{u^t>\gamma\}}g(|\nabla u|)|\nabla u^t|\d x-\int_{\{u>\gamma\}}g(|\nabla u|)|\nabla u|\d x-\left(\int_{\{u^t>\gamma\}}h(|\nabla u|)\d x-\int_{\{u>\gamma\}}h(|\nabla u|)\d x \right).
  \end{split}
  \end{equation}
\end{lemma}

\begin{proof}
  Since $G$ is convex, it follows that $G(b)-G(a)\ge g(a)(b-a)$ for all $a, b\in [0,+\infty)$. Therefore, we have that
  \begin{equation*}
    \begin{split}
        &  \int_{\{u^t>\gamma\}}g(|\nabla u|)|\nabla u^t|\d x-\int_{\{u>\gamma\}}g(|\nabla u|)|\nabla u|\d x \\
         & \le  \int_{\{u^t>\gamma\}}\left(G(|\nabla u^t|)-G(|\nabla u|)+g(|\nabla u|)|\nabla u| \right)\d x-\int_{\{u>\gamma\}}g(|\nabla u|)|\nabla u|\d x\\
         & =\int_{\{u^t>\gamma\}}G(|\nabla u^t|)\d x-\int_{\{u>\gamma\}}G(|\nabla u|)\d x +\left(\int_{\{u^t>\gamma\}}h(|\nabla u|)\d x-\int_{\{u>\gamma\}}h(|\nabla u|)\d x \right),
    \end{split}
  \end{equation*}
  from which \eqref{3-7} follows immediately. The proof is thus complete.
\end{proof}

\begin{lemma}\label{le3}
  There exists a constant $C_0>0$ such that
  \begin{equation*}
    \int_{\{u=\gamma\}}g\left(|\nabla u| \right)\d \mathcal{H}_{N-1}(x)\le C_0
  \end{equation*}
  for all sufficiently small $\gamma>0$.  Here $\mathcal{H}_{N-1}$ denotes the $(N-1)$-dimensional Hausdorff measure.
\end{lemma}

\begin{proof}By \eqref{eq6} or \eqref{eq7}, we observe that $|\nabla u|$ does not vanish near the boundary $\partial \Omega$, that is, there exists an open set $\mathcal{V}\supset \partial \Omega$ such that $|\nabla u|>0$ in $\mathcal{V}\cap \Omega$. Since $u\in C(\overline{\Omega})$ and $u>0$ in $\Omega$, it follows that there exists $\gamma_0>0$ such that $|\nabla u|>0$ in $\{0<u<\gamma_0\}$. By the implicit function theorem, we have that for every $\gamma\in (0, \gamma_0)$, $\{u>\gamma\}$ is an open subset of $\Omega$ with $\partial \{u>\gamma\}=\{u=\gamma\}$, and $\{u=\gamma\}$ is locally a $C^1$-hypersurface. Integrating \eqref{eq3} over $\{u>\gamma\}$, and applying Green's theorem (see, e.g., \cite{EvansMR3409135, TemamMR1846644}), we obtain
\begin{equation*}
  \int_{\{u=\gamma\}}g(|\nabla u|) \d \mathcal{H}_{N-1}(x)=\int_{\{u>\gamma\}}f(u)\d x\le \int_{\Omega} \left|f(u)\right|\d x\le C_0
\end{equation*}
for some constant $C_0>0$ independent of $\gamma$. This completes the proof.
\end{proof}

\begin{lemma}\label{le4}
  Let $\varepsilon > 0$ be given. Then there exists $\delta_0 > 0$ such that for any $\gamma \in (0, \delta_0)$, the following holds:
  \begin{equation*}
     \int_{\{u^t>\gamma\}}h(|\nabla u|)\d x-\int_{\{u>\gamma\}}h(|\nabla u|)\d x  \le \varepsilon t
  \end{equation*}
 for all sufficiently small $t\ge 0$.
\end{lemma}

\begin{proof}
Fix $\gamma>0$. By the definition of CStS, we have
\begin{equation}\label{add1}
  \{u^t>\gamma\} \bigtriangleup \{u>\gamma\} \subset \{x\in \R^N: \operatorname {dist}(x, \{u=\gamma\})<C_1t \},\ \ \ \forall\, t\ge 0,
\end{equation}
for some constant $C_1>0$ depending only on $u$. Here the symbol $\bigtriangleup$ means the symmetric difference. Recall that $u$ is Lipschitz continuous in $\R^N$. Denote its Lipschitz constant by $L$. It follows from \eqref{add1} that
\begin{equation}\label{add2}
\begin{split}
  \{u^t>\gamma\} \bigtriangleup \{u>\gamma\} & \subset \{x\in \R^N: \operatorname {dist}(x, \{u=\gamma\})<C_0t \} \\
     &   \subset \{\gamma-C_1Lt\le u\le \gamma+C_1Lt \}\subset \{\gamma/2\le u \le 2\gamma\}
\end{split}
\end{equation}
for all $t\in [0, \gamma/(2C_1L)]$. Here, $\operatorname {dist}(x, \{u=\gamma\}):=\inf\{|x-y|: y\in \{u=\gamma\}\}$. For clarity, we continue our proof by dealing with two cases separately in the sequel.

\textbf{Case 1}: Let $u$ be as given in Theorem \ref{th1}. Since $\mathcal{L}^N(\{u^t>\gamma\})=\mathcal{L}^N(\{u>\gamma\})$, we have
\begin{equation}\label{add3}
\begin{split}
      & \int_{\{u^t>\gamma\}}h(|\nabla u|)\d x-\int_{\{u>\gamma\}}h(|\nabla u|)\d x \\
     & =\int_{\{u^t>\gamma\}}\left[h(|\nabla u|)-h(\textbf{c})\right]\d x-\int_{\{u>\gamma\}}\left[h(|\nabla u|)-h(\textbf{c})\right]\d x\\
     &\le \int_{ \{u^t>\gamma\} \bigtriangleup \{u>\gamma\} }\left|h(|\nabla u|)-h(\textbf{c})\right|\d x.
\end{split}
\end{equation}
There exists $\delta_1>0$ such that for any $\gamma\in (0, \delta_1)$, we have
\begin{equation}\label{add4}
  |\nabla u(x)|\ge \textbf{c}/2,\ \ \ \left|h(|\nabla u(x)|)-h(\textbf{c})\right|\le \frac{g(\textbf{c}/2)\textbf{c}\varepsilon}{5C_0C_1L}
\end{equation}
for all $x\in \{\gamma/2\le u \le 2\gamma\}$, where $C_0$ is as given in Lemma \ref{le3}. By applying the co-area formula (see \cite{EvansMR3409135}) and combining \eqref{add2}, \eqref{add3}, and \eqref{add4}, we obtain that for any sufficiently small $\gamma$, the following holds:
\begin{equation*}
\begin{split}
    & \int_{\{u^t>\gamma\}}h(|\nabla u|)\d x-\int_{\{u>\gamma\}}h(|\nabla u|)\d x \\
     &\le \frac{g(\textbf{c}/2)\textbf{c}\varepsilon}{5C_0C_1L} \int_{\gamma+C_1Lt}^{\gamma-C_1Lt}\d s \int_{\{u=s\}}\frac{1}{|\nabla u(x)|}\d \mathcal{H}_{N-1}(x) \\
     & \le \frac{g(\textbf{c}/2)\textbf{c}\varepsilon}{5C_0C_1L} \cdot \frac{2}{g(\textbf{c}/2)\textbf{c}}\int_{\gamma+C_1Lt}^{\gamma-C_1Lt}\d s \int_{\{u=s\}}g(|\nabla u(x)|)\d \mathcal{H}_{N-1}(x)\\
     & \le \frac{g(\textbf{c}/2)\textbf{c}\varepsilon}{5C_0C_1L} \cdot \frac{2}{g(\textbf{c}/2)\textbf{c}}\cdot 2C_1Lt\cdot C_0\le \varepsilon t
\end{split}
\end{equation*}
for all sufficiently small $t\ge 0$. This completes the proof of this case.

\smallskip
\textbf{Case 2}: Let $u$ be as given in Theorem \ref{th2}. Recalling \eqref{eq7}, there exists $\delta_2>0$ such that for any $\gamma\in (0, \delta_2)$, we have
\begin{equation*}
  |\nabla u(x)|\le \varepsilon/(2C_0C_1)
\end{equation*}
for all $x\in \{\gamma/2\le u \le 2\gamma\}$, where $C_0$ is as given in Lemma \ref{le3}. By the mean value theorem, it follows that for any $\gamma\in (0, \delta_2)$,
\begin{equation}\label{add5}
\begin{split}
    \{u^t>\gamma\} \bigtriangleup \{u>\gamma\} & \subset \{x\in \R^N: \operatorname {dist}(x, \{u=\gamma\})<C_1t \} \\
     & \subset \{\gamma-{\varepsilon t}/{(2C_0)}\le u\le \gamma+{\varepsilon t}/{(2C_0)} \}
\end{split}
\end{equation}
for all $t\in [0, \gamma/(2C_1L)]$. By applying the co-area formula and combining Lemma \ref{le3} with \eqref{add5}, we obtain that for any sufficiently small $\gamma$, the following holds:
\begin{equation*}
\begin{split}
   \int_{\{u^t>\gamma\}}h(|\nabla u|)\d x-\int_{\{u>\gamma\}}h(|\nabla u|)\d x &\le \int_{ \{u^t>\gamma\} \bigtriangleup \{u>\gamma\} } |\nabla u|g(|\nabla u|)\d x\\
     &\le \int_{\gamma-{\varepsilon t}/{(2C_0)}}^{\gamma+{\varepsilon t}/{(2C_0)}}\d s \int_{\{u=s\}}g\left(|\nabla u| \right)\d \mathcal{H}_{N-1}(x)\\
     &\le C_0\cdot \frac{\varepsilon t}{2C_0}\cdot 2 \le \varepsilon t
\end{split}
\end{equation*}
for all sufficiently small $t\ge 0$. This completes the proof of this case, thereby concluding the proof of the lemma.
\end{proof}

\begin{lemma}\label{le5}
 Let $f$ be a function satisfying assumption $(\mathcal{A})$. Let $\varepsilon > 0$ be given. Then there exists $\delta_0 > 0$ such that for any $\gamma \in (0, \delta_0)$, the following holds:
 \begin{equation*}
   \int_{\Omega}f(u)\left[(u-\gamma)_+^t- (u-\gamma)_+\right]\d x\ge -\varepsilon t+o(t)
 \end{equation*}
for all $t\ge 0$. Here the symbol $o(t)$ denotes any quantities satisfying $\lim_{t\to 0} o(t)/t=0$.
\end{lemma}

\begin{proof}
It suffices to consider $\gamma\in (0, \sup_{\R^N}u)$. Recall that $u$ is Lipschitz continuous in $\R^N$. Denote its Lipschitz constant by $L$. For any $\gamma>0$, the function $(u-\gamma)_+$ has Lipschitz constant less than or equal to $L$. In view of Proposition \ref{pro0} (10), there exists a constant $C_0>0$ such that
\begin{equation}\label{3-9}
  |(u(x)-\gamma)_+^t-(u(x)-\gamma)_+|\le C_0t,\ \ \ \forall\, x\in \R^N,\ \gamma>0,\ t\ge 0.
\end{equation}
 We calculate that
\begin{equation}\label{3-10}
\begin{split}
    &  \int_{\Omega}f(u)\left[(u-\gamma)_+^t- (u-\gamma)_+\right]\d x= \int_{\Omega}\left[f(u)-f((u-\gamma)_++\gamma )\right]\left[(u-\gamma)_+^t- (u-\gamma)_+\right]\d x \\
     &\ \ \  +  \int_{\Omega}f((u-\gamma)_++\gamma )\left[(u-\gamma)_+^t- (u-\gamma)_+\right]\d x=:I_1(t)+I_2(t).
\end{split}
\end{equation}
Set $C_1:=\sup\left\{|f(z)|: z\in[0, \sup_{\R^N} u]\right\}$. For $I_1(t)$, by \eqref{3-9} we have
\begin{equation*}
\begin{split}
    \left|I_1(t)\right| & = \int_{\Omega}\left|f(u)-f((u-\gamma)_++\gamma )\right|\left|(u-\gamma)_+^t- (u-\gamma)_+\right|\d x \\
     & \le 2C_1C_0 t \mathcal{L}^N\left(\{0<u<\gamma\}\right).
\end{split}
\end{equation*}
We now turn to $I_2(t)$. We will show that $I_2(t)\ge o(t)$ as $t\to 0$. To simplify the notation, we set $\tilde{f}(z):=f(z+\gamma)$, for $z\in[0, +\infty)$, and $v:=(u-\gamma)_+$. Recalling assumption $(\mathcal{A})$, we apply Jordan's decomposition theorem to obtain the decomposition $\tilde{f} = \tilde{f}_1+\tilde{f}_2+\tilde{f}_3$, where $\tilde{f}_1$ is continuous, $\tilde{f}_2$ is nondecreasing, and $\tilde{f}_3$ is nonincreasing. Moreover, we may assume that $\tilde{f}_i$ are all bounded and $f_2(0)=0$, since $u$ is bounded in $\Omega$. Hence, $I_2(t)$ can be split into
\begin{equation*}
I_2(t)=\int_{\Omega}\tilde{f}(v)\left(v^t- v\right)\d x =\sum_{i=1}^{3}\int_{\Omega}\tilde{f}_i(v)\left(v^t- v\right)\d x =:\sum_{i=1}^{3}I_{2}^{(i)}(t).
\end{equation*}
Set $\tilde{F}_i(z)=\int_{0}^{z}\tilde{f}_i(s)\d s,\ i=1, 2, 3$. Observe that the support of $v^t$ is contained in $\Omega$ for all sufficiently small $t$. By Proposition \ref{pro0} (5), we have that
\begin{equation*}
  0=\int_\Omega \tilde{F}_i(v^t)\d x-\int_\Omega \tilde{F}_i(v^t)\d x=\int_\Omega\left( \int_{0}^{1}\tilde{f}_i\left(v(x)+\theta (v^t(x)-v(x)) \right)    \right)   \left(v^t(x)- v(x)\right) \d x.
\end{equation*}
Let us first consider $I_{2}^{(1)}(t)$. We have
\begin{equation}\label{3-12}
\begin{split}
 & I_{2}^{(1)}(t)  =\int_{\Omega}\tilde{f}_1(v)\left(v^t- v\right)\d x \\
     &=\int_{\Omega}\tilde{f}_1(v)\left(v^t- v\right)\d x- \int_\Omega\left( \int_{0}^{1}\tilde{f}_1\left(v(x)+\theta (v^t(x)-v(x)) \right) \d \theta   \right)   \left(v^t(x)- v(x)\right) \d x \\
     &= \int_{\Omega} \left[\int_{0}^{1}\left(\tilde{f}_1(v(x))- \tilde{f}_1\left(v(x)+\theta (v^t(x)-v(x))  \right) \right)\d \theta  \right] \left(v^t(x)- v(x)\right) \d x\\
     & \ge - \sup\left\{|\tilde{f}_1(v(x))- \tilde{f}_1\left(v(x)+\theta (v^t(x)-v(x))  \right)|: x\in \Omega,\ \theta\in [0, 1]  \right\} \cdot C_0t\cdot \mathcal{L}^N(\Omega)\\
     &\ge o(t),\ \ \ \text{as}\ t\to 0.
\end{split}
\end{equation}
Similarly, for $I_{2}^{(2)}(t)$, we have
\begin{equation}\label{3-13}
\begin{split}
 & I_{2}^{(2)}(t)   =\int_{\Omega}\tilde{f}_2(v)\left(v^t- v\right)\d x \\
     &=\int_{\Omega}\tilde{f}_2(v)\left(v^t- v\right)\d x- \int_\Omega\left( \int_{0}^{1}\tilde{f}_2\left(v(x)+\theta (v^t(x)-v(x)) \right) \d \theta   \right)   \left(v^t(x)- v(x)\right) \d x \\
     &= \int_{\Omega} \left[\tilde{f}_2(v(x))-\int_{0}^{1} \tilde{f}_2\left(v(x)+\theta (v^t(x)-v(x))  \right) \d \theta  \right] \left(v^t(x)- v(x)\right) \d x\\
     & \ge - \int_{\Omega} \left|\tilde{f}_2(v^t(x))- \tilde{f}_2(v(x))\right|\left|v^t(x)- v(x)\right| \d x\\
     & \ge -\|\tilde{f}_2(v^t)-\tilde{f}_2(v)\|_{L^1(\Omega)} \cdot C_0 t\cdot \mathcal{L}^N(\Omega) \\
     &\ge o(t),\ \ \ \text{as}\ t\to 0,
\end{split}
\end{equation}
where we have used that the monotonicity of $\tilde{f}_2$ and the fact that $\|\tilde{f}_2(v^t)-\tilde{f}_2(v)\|_{L^1(\Omega)}\to 0$ as $t\to 0$, due to Proposition \ref{pro0} (3) and (6). Finally, for $I_{2}^{(3)}(t)$, we have

\begin{equation}\label{3-14}
\begin{split}
 & I_{2}^{(3)}(t)  =\int_{\Omega}\tilde{f}_3(v)\left(v^t- v\right)\d x \\
     &=\int_{\Omega}\tilde{f}_3(v)\left(v^t- v\right)\d x- \int_\Omega\left( \int_{0}^{1}\tilde{f}_3\left(v(x)+\theta (v^t(x)-v(x)) \right) \d \theta   \right)   \left(v^t(x)- v(x)\right) \d x \\
     &= \int_{\Omega} \left[\tilde{f}_3(v(x))-\int_{0}^{1} \tilde{f}_3\left(v(x)+\theta (v^t(x)-v(x))  \right) \d \theta  \right] \left(v^t(x)- v(x)\right) \d x\\
     & \ge 0,
\end{split}
\end{equation}
where we have used the fact that $\tilde{f}_3$ is nonincreasing. Combining \eqref{3-12}, \eqref{3-13} and \eqref{3-14}, we conclude that $I_2(t)\ge o(t)$ as $t\to 0$. Therefore, we arrive
\begin{equation}\label{3-15}
\begin{split}
    \int_{\Omega}f(u)\left[(u-\gamma)_+^t- (u-\gamma)_+\right]\d x &=I_1(t)+I_2(t) \\
     & \ge -2C_1C_0 t \mathcal{L}^N\left(\{0<u<\gamma\}\right)+o(t).
\end{split}
\end{equation}
Note that $\mathcal{L}^N\left(\{0<u<\gamma\}\right)\to 0$ as $\gamma \to 0$. For some $\delta_0 > 0$, it holds that for any $\gamma \in (0, \delta_0)$,
\begin{equation*}
  2C_1C_0\mathcal{L}^N\left(\{0<u<\gamma\}\right)<\varepsilon.
\end{equation*}
Thus, by \eqref{3-15}, we deduce that
 \begin{equation*}
   \int_{\Omega}f(u)\left[(u-\gamma)_+^t- (u-\gamma)_+\right]\d x\ge -\varepsilon t+o(t)
 \end{equation*}
for all $t\ge 0$. The proof is thereby complete.
\end{proof}

\subsection{Proofs of Theorems \ref{th1} and \ref{th2}}

We are now in a position to prove Theorems \ref{th1} and \ref{th2}.

\smallskip

\noindent \textbf{Proof of Theorems \ref{th1} and \ref{th2}}:

As explained at the beginning of this section, it suffices to prove that $u$ is locally symmetric. This can be achieved if we can show that $(u - \gamma_0)_+$ is locally symmetric for any $\gamma_0>0$. We will show that $(u - \gamma_0)_+$ is locally symmetric in the direction $x_1$. The local symmetry in other directions can be established in a similar manner. Recalling Definition \ref{csts}, let $(u - \gamma_0)_+^t$ denote the CStS of $(u - \gamma_0)_+$ with respect to $x_1$. Thanks to Proposition \ref{pro2}, it suffices to prove that
\begin{equation*}
   \int_{\R^N}G\left(|\nabla (u-\gamma_0)_+^t|\right)\d x-\int_{\R^N}G\left(|\nabla (u-\gamma_0)_+|\right)\d x= o(t)
\end{equation*}
as $t\to 0$. By Proposition \ref{pro0} (10), we have
\begin{equation*}
  \int_{\R^N}G\left(|\nabla (u-\gamma_0)_+^t|\right)\d x-\int_{\R^N}G\left(|\nabla (u-\gamma_0)_+|\right)\d x\le 0.
\end{equation*}
So it is reduced to proving that
\begin{equation*}
  \int_{\R^N}G\left(|\nabla (u-\gamma_0)_+^t|\right)\d x-\int_{\R^N}G\left(|\nabla (u-\gamma_0)_+|\right)\d x\ge o(t)
\end{equation*}
as $t\to 0$. This can be proven by showing that for any $\varepsilon>0$, the following inequality holds:
\begin{equation}\label{3-20}
 \int_{\R^N}G\left(|\nabla (u-\gamma_0)_+^t|\right)\d x-\int_{\R^N}G\left(|\nabla (u-\gamma_0)_+|\right)\d x\ge -\varepsilon t
\end{equation}
for all sufficiently small $t$. By Lemma \ref{le1}, \eqref{3-20} can be established by showing that there exists $\gamma_1\in(0, \gamma_0)$ such that
\begin{equation}\label{3-23}
   \int_{\R^N}G\left(|\nabla (u-\gamma_1)_+^t|\right)\d x-\int_{\R^N}G\left(|\nabla (u-\gamma_1)_+|\right)\d x\ge -\varepsilon t
\end{equation}
for all sufficiently small $t$. Next, we focus on demonstrating the existence of such a $\gamma_1$. Note that $[(u-\gamma_1)_+^t-(u-\gamma_1)_+]\in H_0^1(\Omega)$ if $t$ is small enough. Recalling \eqref{eq4}, let $[(u-\gamma_1)_+^t-(u-\gamma_1)_+]$ be a test function, we then obtain
\begin{equation*}
  \int_\Omega g(|\nabla u|)\frac{\nabla u \cdot \nabla[(u-\gamma_1)_+^t-(u-\gamma_1)_+]}{|\nabla u|}\,\d x=\int_{\Omega} f(u)[(u-\gamma_1)_+^t-(u-\gamma_1)_+]\,\d x,
\end{equation*}
where the integrand on the left-hand side is understood to be zero at each point $x$ where $\nabla u(x)=0$. It follows that
\begin{equation}\label{3-24}
  \int_{\{u^t>\gamma_1\}} g(|\nabla u|)|\nabla u^t|\,\d x - \int_{\{u>\gamma_1\}} g(|\nabla u|)|\nabla u|\,\d x\ge \int_{\Omega} f(u)[(u-\gamma_1)_+^t-(u-\gamma_1)_+]\,\d x.
\end{equation}
By Lemmas \ref{le4} and \ref{le5}, we deduce that there exists $\delta_0 > 0$ such that for any $\gamma \in (0, \delta_0)$, the following hold:
 \begin{equation}\label{3-25}
 \begin{split}
   & \int_{\{u^t>\gamma\}}h(|\nabla u|)\d x-\int_{\{u>\gamma\}}h(|\nabla u|)\d x  \le \frac{\varepsilon}{3} t, \\
    & \int_{\Omega}f(u)\left[(u-\gamma)_+^t- (u-\gamma)_+\right]\d x \ge -\frac{\varepsilon}{3} t+o(t)
 \end{split}
 \end{equation}
for all sufficiently small $t$. Now, let us fix $\gamma_1<\min\{\delta_0, \gamma_0\}$ so that \eqref{3-25} holds with $\gamma=\gamma_1$. Combining \eqref{3-24}, \eqref{3-25}, and Lemma \ref{le2}, we derive that
\begin{equation*}
  \begin{split}
  \int_{\R^N}&G\left(|\nabla (u-\gamma_1)_+^t|\right)\d x-\int_{\R^N}G\left(|\nabla (u-\gamma_1)_+|\right)\d x\\
      &  = \int_{\{u^t>\gamma_1\}}G(|\nabla u^t|)\d x-\int_{\{u>\gamma_1\}}G(|\nabla u|)\d x \\
       & \ge \int_{\{u^t>\gamma_1\}}g(|\nabla u|)|\nabla u^t|\d x-\int_{\{u>\gamma_1\}}g(|\nabla u|)|\nabla u|\d x\\
       &\ \ \ \ \ \ \ \ \ -\left(\int_{\{u^t>\gamma_1\}}h(|\nabla u|)\d x-\int_{\{u>\gamma_1\}}h(|\nabla u|)\d x \right)\\
       &   \ge  -\frac{\varepsilon}{3} t+o(t)-\frac{\varepsilon}{3} t=-\frac{2}{3}\varepsilon t+o(t)\ge -\varepsilon t
  \end{split}
  \end{equation*}
for all sufficiently small $t$, which implies \eqref{3-23}. The proof is thereby complete.

\section{Proof of Theorem \ref{th3}}\label{s4}

This section is dedicated to the proof of Theorem \ref{th3}. Let $u$ be as given in Theorem \ref{th3}. For convenience, we extend $u$ by setting $u\equiv \eta$ in $\overline{\Omega_1}$, and $u\equiv 0$ in $\R^N\backslash\Omega_0$, while retaining the notation $u$. It follows from \eqref{eq14} that $u$ is Lipschitz in $\R^N$.

Our goal is to show that both $\Omega_0$ and $\Omega_1$ are balls, which may not be concentric. Without loss of generality, assume that $\eta = 1$. Observe that if $u$ is locally symmetric, then the super-level set $\Omega_0=\{u>0\}$ consists of a countable union of mutually disjoint balls. Since $\Omega_0$ is connected, it must be a single ball. On the other hand, for each positive integer $n$, the super-level set $\{u>1-1/n\}$ consists of at most countably many disjoint open balls, among which there is a unique open ball that contains $\Omega_1$, denoted by $B^{(n)}$. Clearly, we have $B^{(1)}\supset B^{(2)}\supset \cdots \supset B^{(n)}\supset \cdots$. We check that $\overline{\Omega_1}=\cap_{n=1}^\infty B^{(n)}$ is a closed ball, which, by the definition of a ring-shaped domain, implies that $\Omega_1$ must be an open ball. Thus, it remains to establish the local symmetry of $u$. This can be achieved using a method analogous to that of the previous section, with necessary modifications. Since the ring-shaped domain $\Omega$ has two boundaries, namely the inner boundary $\partial \Omega_1$ and the outer boundary $\partial \Omega_0$, truncation must be performed on both. Apart from this, no significant differences arise compared to the previous section. Hence, we only sketch the proof, as the details follow from straightforward modifications based on the ideas developed earlier.


As in the preceding section, we first establish several auxiliary technical lemmas. Since their proofs closely follow those in the previous section, we omit the details.

Recalling Definition \ref{csts}, let $u^t$ denote the CStS of $u$ with respect to $x_1$. Let $G$ and $h$ be defined as in the preceding section. We define the truncation operator as follows:
\begin{equation*}
  \mathcal{T}_\gamma^\beta[u]=\min\{\beta, (u-\gamma)_+\},\ \ \ \beta>\gamma>0.
\end{equation*}

\begin{lemma}[Monotonicity lemma]\label{le6}
Let $\beta_1, \beta_0, \gamma_0$, and $\gamma_1$ be positive constants satisfying $\beta_1\ge \beta_0\ge \gamma_0\ge \gamma_1$. Then
\begin{equation*}
\begin{split}
     \int_{\R^N}G\left(\left|\nabla\left(\mathcal{T}_{\gamma_0}^{\beta_0}[u]\right)^t\right|\right)\d x& -\int_{\R^N}G\left(\left|\nabla\mathcal{T}_{\gamma_0}^{\beta_0}[u]\right|\right)\d x \\
     & \ge  \int_{\R^N}G\left(\left|\nabla\left(\mathcal{T}_{\gamma_1}^{\beta_1}[u]\right)^t\right|\right)\d x -\int_{\R^N}G\left(\left|\nabla\mathcal{T}_{\gamma_1}^{\beta_1}[u]\right|\right)\d x
\end{split}
\end{equation*}
for all $t\ge0$.
\end{lemma}

\begin{lemma}\label{le7}
  Let $\beta$ and $\gamma$ be two fixed positive constants satisfying $\beta>\gamma$. Then
  \begin{equation}\label{4-7}
  \begin{split}
        \int_{\{\beta>u^t>\gamma\}} &G(|\nabla u^t|)\d x-\int_{\{\beta>u>\gamma\}}G(|\nabla u|)\d x \\
       &  \ge\left(\int_{\{\beta>u^t>\gamma\}}g(|\nabla u|)|\nabla u^t|\d x-\int_{\{\beta>u>\gamma\}}g(|\nabla u|)|\nabla u|\d x\right)\\
       &\ \ \ \ \ \ -\left(\int_{\{\beta>u^t>\gamma\}}h(|\nabla u|)\d x-\int_{\{\beta>u>\gamma\}}h(|\nabla u|)\d x \right).
  \end{split}
  \end{equation}
\end{lemma}

\begin{lemma}\label{le8}
  There exist constants $C_0>0$ and $\delta_0>0$ such that
  \begin{equation*}
    \int_{\{u=\gamma\}}g\left(|\nabla u| \right)\d \mathcal{H}_{N-1}(x)\le C_0
  \end{equation*}
  for all positive constant $\gamma$ satisfying $\gamma<\delta_0$ or $0<\eta-\gamma<\delta_0$.
\end{lemma}

\begin{lemma}\label{le9}
  Let $\varepsilon > 0$ be given. Then there exists a small constant $\delta_0 > 0$ such that for any $\beta\in (\eta-\delta_0, \eta)$ and $\gamma \in (0, \delta_0)$, the following holds:
  \begin{equation*}
     \int_{\{\beta>u^t>\gamma\}}h(|\nabla u|)\d x-\int_{\{\beta>u>\gamma\}}h(|\nabla u|)\d x  \le \varepsilon t
  \end{equation*}
 for all sufficiently small $t\ge 0$.
\end{lemma}

\begin{lemma}\label{le10}
 Let $f$ be a function satisfying assumption $(\mathcal{A})$. Let $\varepsilon > 0$ be given. Then there exists a small constant $\delta_0 > 0$ such that for any $\beta\in (\eta-\delta_0, \eta)$ and $\gamma \in (0, \delta_0)$, the following holds:
 \begin{equation*}
   \int_{\Omega}f(u)\left[\left(\mathcal{T}_{\gamma}^{\beta}[u]\right)^t-\mathcal{T}_{\gamma}^{\beta}[u]\right]\d x\ge -\varepsilon t+o(t)
 \end{equation*}
for all $t\ge 0$.
\end{lemma}

With the above technical lemmas established, we are now ready to prove Theorem \ref{th3}.

\smallskip

\noindent \textbf{Proof of Theorem \ref{th3}}:

The proof closely parallels that of the previous section. Therefore, we provide only a sketch, omitting the details. As explained at the beginning of this section, it suffices to prove that $u$ is locally symmetric. This can be achieved if we can show that $\mathcal{T}_{\gamma_0}^{\beta_0}[u]$ is locally symmetric for any $\beta_0>\gamma_0>0$. It suffices to show that $\mathcal{T}_{\gamma_0}^{\beta_0}[u]$ is locally symmetric in the direction $x_1$. The local symmetry in other directions can be established in a similar manner. The key idea is to apply Proposition \ref{pro2}. This can be accomplished through the following steps. First, thanks to Lemma \ref{le6}, we reduce the problem to finding appropriate values of $\beta_1$ and $\gamma_1$. Next, by Lemma \ref{le7}, the task is transformed into estimating the two terms on the right-hand side of \eqref{4-7}. Then, testing equation \eqref{eq3} with $\left(\left(\mathcal{T}_{\gamma_1}^{\beta_1}[u]\right)^t-\mathcal{T}_{\gamma_1}^{\beta_1}[u]\right)$ and using Lemma \ref{le10}, we derive the desired estimate for the first term on the right-hand side of \eqref{4-7}. Finally, Lemma \ref{le9} provides the desired estimate for the second term on the right-hand side of \eqref{4-7}. These steps complete the proof.

\begin{remark}\label{remark-4}
  If $\Omega_i$ (for $i = 1$ or $2$) is known a priori to be a ball, then the Neumann boundary condition for $u$ on $\partial\Omega_i$ becomes redundant and may be omitted. This fact can be obtained by a slight modification of the above proof. In fact, the truncation argument near the boundary is primarily employed to ensure the legitimacy of the test function. If $\Omega_i$ is already known a priori to be a ball, truncation near its boundary is unnecessary. For instance, suppose that $\Omega_1$ is a ball, say $B_{r_1}$ (after a suitable translation, if necessary). Then, for all $\gamma > 0$ and $t \ge 0$, we have $[(u - \gamma)_+^t - (u - \gamma)_+] \equiv 0$ in $\overline{\Omega_1}$. It follows that $[(u - \gamma)_+^t - (u - \gamma)_+] \in H^1_0(\Omega)$ provided that $t \ge 0$ is sufficiently small. The problem reduces to showing that $(u - \gamma)_+$ is locally symmetric for any $\gamma > 0$. This requires only replacing the test function $\left(\left(\mathcal{T}_{\gamma_1}^{\beta_1}[u]\right)^t-\mathcal{T}_{\gamma_1}^{\beta_1}[u]\right)$ in the preceding argument with $[(u - \gamma)_+^t - (u - \gamma)_+]$, and then carrying out a similar discussion.

\end{remark}

\section*{Acknowledgements}

The authors are grateful to Professor Reichel for kindly sharing his paper \cite{ReichelMR1416582}, and to Y. R.-Y. Zhang for helpful discussions and comments. The research of D. Cao is supported by National Key R \& D Program (2023YFA1010001) and NNSF of China (grant No. 12371212). The research of J. Wei is partially supported by GRF from RGC of Hong Kong entitled ``New frontiers in singularity formations in nonlinear partial differential equations". The research of W. Zhan is supported by NNSF of China (grant No. 12201525).

\bigskip
\noindent \textbf{Data Availability}: Data sharing is not applicable to this article as no datasets were generated or analyzed during the current study.

\smallskip
\noindent \textbf{Conflict of interest}: The authors declare that they have no conflict of interest.

\bibliography{ref}

\begin{thebibliography}{10}

\bibitem{Agostiniani2024}
{\sc V.~Agostiniani, S.~Borghini, and L.~Mazzieri}, {\em On the {S}errin
  problem for ring-shaped domains}, to appear in J. Eur. Math. Soc. (JEMS).

\bibitem{AlexandrovMR102114}
{\sc A.~D. Aleksandrov}, {\em Uniqueness theorems for surfaces in the large.
  {V}}, Vestnik Leningrad. Univ., 13 (1958), pp.~5--8.

\bibitem{AlexandrovMR143162}
{\sc A.~D. Alexandrov}, {\em A characteristic property of spheres}, Ann. Mat.
  Pura Appl. (4), 58 (1962), pp.~303--315.

\bibitem{BorghiniMR4431667}
{\sc S.~Borghini}, {\em Symmetry results for {S}errin-type problems in doubly
  connected domains}, Math. Eng., 5 (2023), pp.~Paper No. 027, 16.

\bibitem{BrandoliniMR2436453}
{\sc B.~Brandolini, C.~Nitsch, P.~Salani, and C.~Trombetti}, {\em On the
  stability of the {S}errin problem}, J. Differential Equations, 245 (2008),
  pp.~1566--1583.

\bibitem{BrandoliniMR2448319}
\leavevmode\vrule height 2pt depth -1.6pt width 23pt, {\em Serrin-type
  overdetermined problems: an alternative proof}, Arch. Ration. Mech. Anal.,
  190 (2008), pp.~267--280.

\bibitem{Bro1995MR1330619}
{\sc F.~Brock}, {\em Continuous {S}teiner-symmetrization}, Math. Nachr., 172
  (1995), pp.~25--48.

\bibitem{Broc2000MR1758811}
\leavevmode\vrule height 2pt depth -1.6pt width 23pt, {\em Continuous
  rearrangement and symmetry of solutions of elliptic problems}, Proc. Indian
  Acad. Sci. Math. Sci., 110 (2000), pp.~157--204.

\bibitem{Bro2007MR2569330}
\leavevmode\vrule height 2pt depth -1.6pt width 23pt, {\em Rearrangements and
  applications to symmetry problems in {PDE}}, in Handbook of differential
  equations: stationary partial differential equations. {V}ol. {IV}, Handb.
  Differ. Equ., Elsevier/North-Holland, Amsterdam, 2007, pp.~1--60.

\bibitem{Bro2016MR3509374}
\leavevmode\vrule height 2pt depth -1.6pt width 23pt, {\em Symmetry for a
  general class of overdetermined elliptic problems}, NoDEA Nonlinear
  Differential Equations Appl., 23 (2016), pp.~Art. 36, 16.

\bibitem{BrockMR1947461}
{\sc F.~Brock and A.~Henrot}, {\em A symmetry result for an overdetermined
  elliptic problem using continuous rearrangement and domain derivative}, Rend.
  Circ. Mat. Palermo (2), 51 (2002), pp.~375--390.

\bibitem{ButtazzoMR2764863}
{\sc G.~Buttazzo and B.~Kawohl}, {\em Overdetermined boundary value problems
  for the {$\infty$}-{L}aplacian}, Int. Math. Res. Not. IMRN,  (2011),
  pp.~237--247.

\bibitem{CavallinaMR4340795}
{\sc L.~Cavallina}, {\em The simultaneous asymmetric perturbation method for
  overdetermined free boundary problems}, Nonlinear Anal., 215 (2022),
  pp.~Paper No. 112685, 17.

\bibitem{ChoulliMR1626395}
{\sc M.~Choulli and A.~Henrot}, {\em Use of the domain derivative to prove
  symmetry results in partial differential equations}, Math. Nachr., 192
  (1998), pp.~91--103.

\bibitem{CianchiMR2545870}
{\sc A.~Cianchi and P.~Salani}, {\em Overdetermined anisotropic elliptic
  problems}, Math. Ann., 345 (2009), pp.~859--881.

\bibitem{CiraoloMR3959271}
{\sc G.~Ciraolo and L.~Vezzoni}, {\em On {S}errin's overdetermined problem in
  space forms}, Manuscripta Math., 159 (2019), pp.~445--452.

\bibitem{CortMR1387457}
{\sc C.~Cort\'azar, M.~Elgueta, and P.~Felmer}, {\em Symmetry in an elliptic
  problem and the blow-up set of a quasilinear heat equation}, Comm. Partial
  Differential Equations, 21 (1996), pp.~507--520.

\bibitem{DiBenedettoMR709038}
{\sc E.~DiBenedetto}, {\em {$C\sp{1+\alpha }$}\ local regularity of weak
  solutions of degenerate elliptic equations}, Nonlinear Anal., 7 (1983),
  pp.~827--850.

\bibitem{Drivas2024}
{\sc T.~D. Drivas and M.~Nualart}, {\em A geometric characterization of steady
  laminar flow}, arXiv:2410.18946.

\bibitem{Enci2023}
{\sc A.~Enciso, A.~J. Fernández, D.~Ruiz, and P.~Sicbaldi}, {\em A
  schiffer-type problem for annuli with applications to stationary planar
  {E}uler flows}, to appear in Duke Math. J.

\bibitem{EvansMR3409135}
{\sc L.~C. Evans and R.~F. Gariepy}, {\em Measure theory and fine properties of
  functions}, Textbooks in Mathematics, CRC Press, Boca Raton, FL, revised~ed.,
  2015.

\bibitem{FarinaMR2366129}
{\sc A.~Farina and B.~Kawohl}, {\em Remarks on an overdetermined boundary value
  problem}, Calc. Var. Partial Differential Equations, 31 (2008), pp.~351--357.

\bibitem{FarinaMR3145008}
{\sc A.~Farina and E.~Valdinoci}, {\em On partially and globally overdetermined
  problems of elliptic type}, Amer. J. Math., 135 (2013), pp.~1699--1726.

\bibitem{Figalli2024}
{\sc A.~Figalli and R.-Y.~Y. Zhang}, {\em Serrin's overdetermined problem in
  rough domains}, arXiv:2407.02293.

\bibitem{FolMR3086464}
{\sc J.~F\"oldes}, {\em On {S}errin's symmetry result in nonsmooth domains and
  its applications}, Adv. Differential Equations, 18 (2013), pp.~523--548.

\bibitem{FragaMR2863764}
{\sc I.~Fragal\`a}, {\em Symmetry results for overdetermined problems on convex
  domains via {B}runn-{M}inkowski inequalities}, J. Math. Pures Appl. (9), 97
  (2012), pp.~55--65.

\bibitem{FraMR2232009}
{\sc I.~Fragal\`a, F.~Gazzola, and B.~Kawohl}, {\em Overdetermined problems
  with possibly degenerate ellipticity, a geometric approach}, Math. Z., 254
  (2006), pp.~117--132.

\bibitem{GarofaloMR980297}
{\sc N.~Garofalo and J.~L. Lewis}, {\em A symmetry result related to some
  overdetermined boundary value problems}, Amer. J. Math., 111 (1989),
  pp.~9--33.

\bibitem{GazzolaMR2257026}
{\sc F.~Gazzola}, {\em No geometric approach for general overdetermined
  elliptic problems with nonconstant source}, Matematiche (Catania), 60 (2005),
  pp.~259--268.

\bibitem{KamburovMR4200475}
{\sc N.~Kamburov and L.~Sciaraffia}, {\em Nontrivial solutions to {S}errin's
  problem in annular domains}, Ann. Inst. H. Poincar\'e{} C Anal. Non
  Lin\'eaire, 38 (2021), pp.~1--22.

\bibitem{Kaw1985MR810619}
{\sc B.~Kawohl}, {\em Rearrangements and convexity of level sets in {PDE}},
  vol.~1150 of Lecture Notes in Mathematics, Springer-Verlag, Berlin, 1985.

\bibitem{KawohlMR4205793}
{\sc B.~Kawohl and M.~Lucia}, {\em Some results related to {S}chiffer's
  problem}, J. Anal. Math., 142 (2020), pp.~667--696.

\bibitem{KhavinsonMR2174103}
{\sc D.~Khavinson, A.~Y. Solynin, and D.~Vassilev}, {\em Overdetermined
  boundary value problems, quadrature domains and applications}, Comput.
  Methods Funct. Theory, 5 (2005), pp.~19--48.

\bibitem{KumaresanMR1487977}
{\sc S.~Kumaresan and J.~Prajapat}, {\em Serrin's result for hyperbolic space
  and sphere}, Duke Math. J., 91 (1998), pp.~17--28.

\bibitem{LeeMR4566201}
{\sc J.~Lee and K.~Seo}, {\em Overdetermined problems in annular domains with a
  spherical-boundary component in space forms}, J. Inequal. Appl.,  (2023),
  pp.~Paper No. 45, 22.

\bibitem{Lian2025}
{\sc Y.~Lian and P.~Sicbaldi}, {\em Rigidity results for the capillary
  overdetermined problem}, arXiv:2503.14215.

\bibitem{Lie2001MR1817225}
{\sc E.~H. Lieb and M.~Loss}, {\em Analysis}, vol.~14 of Graduate Studies in
  Mathematics, American Mathematical Society, Providence, RI, second~ed., 2001.

\bibitem{MagnaniniMR4041100}
{\sc R.~Magnanini and G.~Poggesi}, {\em On the stability for {A}lexandrov's
  soap bubble theorem}, J. Anal. Math., 139 (2019), pp.~179--205.

\bibitem{NitschMR3802818}
{\sc C.~Nitsch and C.~Trombetti}, {\em The classical overdetermined {S}errin
  problem}, Complex Var. Elliptic Equ., 63 (2018), pp.~1107--1122.

\bibitem{PayneMR1021402}
{\sc L.~E. Payne and P.~W. Schaefer}, {\em Duality theorems in some
  overdetermined boundary value problems}, Math. Methods Appl. Sci., 11 (1989),
  pp.~805--819.

\bibitem{PrajapatMR1487978}
{\sc J.~Prajapat}, {\em Serrin's result for domains with a corner or cusp},
  Duke Math. J., 91 (1998), pp.~29--31.

\bibitem{PucciMR2356201}
{\sc P.~Pucci and J.~Serrin}, {\em The maximum principle}, vol.~73 of Progress
  in Nonlinear Differential Equations and their Applications, Birkh\"auser
  Verlag, Basel, 2007.

\bibitem{ReichelMR1416582}
{\sc W.~Reichel}, {\em Radial symmetry by moving planes for semilinear elliptic
  {BVP}s on annuli and other non-convex domains}, in Elliptic and parabolic
  problems ({P}ont-\`a-{M}ousson, 1994), vol.~325 of Pitman Res. Notes Math.
  Ser., Longman Sci. Tech., Harlow, 1995, pp.~164--182.

\bibitem{ReichelMR1463801}
{\sc W.~Reichel}, {\em Radial symmetry for elliptic boundary-value problems on
  exterior domains}, Arch. Rational Mech. Anal., 137 (1997), pp.~381--394.

\bibitem{RosMR3666566}
{\sc A.~Ros, D.~Ruiz, and P.~Sicbaldi}, {\em A rigidity result for
  overdetermined elliptic problems in the plane}, Comm. Pure Appl. Math., 70
  (2017), pp.~1223--1252.

\bibitem{RosMR4046014}
\leavevmode\vrule height 2pt depth -1.6pt width 23pt, {\em Solutions to
  overdetermined elliptic problems in nontrivial exterior domains}, J. Eur.
  Math. Soc. (JEMS), 22 (2020), pp.~253--281.

\bibitem{RosMR3062759}
{\sc A.~Ros and P.~Sicbaldi}, {\em Geometry and topology of some overdetermined
  elliptic problems}, J. Differential Equations, 255 (2013), pp.~951--977.

\bibitem{RuizMR4575796}
{\sc D.~Ruiz}, {\em Symmetry results for compactly supported steady solutions
  of the 2{D} {E}uler equations}, Arch. Ration. Mech. Anal., 247 (2023),
  pp.~Paper No. 40, 25.

\bibitem{SchaeferMR1971630}
{\sc P.~W. Schaefer}, {\em On nonstandard overdetermined boundary value
  problems}, in Proceedings of the {T}hird {W}orld {C}ongress of {N}onlinear
  {A}nalysts, {P}art 4 ({C}atania, 2000), vol.~47, 2001, pp.~2203--2212.

\bibitem{Ser1971MR333220}
{\sc J.~Serrin}, {\em A symmetry problem in potential theory}, Arch. Rational
  Mech. Anal., 43 (1971), pp.~304--318.

\bibitem{ShahgholianMR2916825}
{\sc H.~Shahgholian}, {\em Diversifications of {S}errin's and related symmetry
  problems}, Complex Var. Elliptic Equ., 57 (2012), pp.~653--665.

\bibitem{SicbaldiMR4559541}
{\sc P.~Sicbaldi}, {\em A short survey on overdetermined elliptic problems in
  unbounded domains}, in Current trends in analysis, its applications and
  computation, Trends Math. Res. Perspect., Birkh\"auser/Springer, Cham, 2022,
  pp.~451--461.

\bibitem{SirakovMR1808026}
{\sc B.~Sirakov}, {\em Symmetry for exterior elliptic problems and two
  conjectures in potential theory}, Ann. Inst. H. Poincar\'e{} C Anal. Non
  Lin\'eaire, 18 (2001), pp.~135--156.

\bibitem{Sirakov2002}
\leavevmode\vrule height 2pt depth -1.6pt width 23pt, {\em Overdetermined
  Elliptic Problems in Physics}, Springer Netherlands, Dordrecht, 2002,
  pp.~273--295.

\bibitem{TalentiMR3503198}
{\sc G.~Talenti}, {\em The art of rearranging}, Milan J. Math., 84 (2016),
  pp.~105--157.

\bibitem{TemamMR1846644}
{\sc R.~Temam}, {\em Navier-{S}tokes equations}, AMS Chelsea Publishing,
  Providence, RI, 2001.
\newblock Theory and numerical analysis, Reprint of the 1984 edition.

\bibitem{TolksdorfMR727034}
{\sc P.~Tolksdorf}, {\em Regularity for a more general class of quasilinear
  elliptic equations}, J. Differential Equations, 51 (1984), pp.~126--150.

\bibitem{TraizetMR3192039}
{\sc M.~Traizet}, {\em Classification of the solutions to an overdetermined
  elliptic problem in the plane}, Geom. Funct. Anal., 24 (2014), pp.~690--720.

\bibitem{VogelMR1200301}
{\sc A.~L. Vogel}, {\em Symmetry and regularity for general regions having a
  solution to certain overdetermined boundary value problems}, Atti Sem. Mat.
  Fis. Univ. Modena, 40 (1992), pp.~443--484.

\bibitem{WangMR3952780}
{\sc K.~Wang and J.~Wei}, {\em On {S}errin's overdetermined problem and a
  conjecture of {B}erestycki, {C}affarelli and {N}irenberg}, Comm. Partial
  Differential Equations, 44 (2019), pp.~837--858.

\bibitem{Wang2023}
{\sc Y.~Wang and W.~Zhan}, {\em On the rigidity of the 2d incompressible euler
  equations}, arXiv:2307.00197v2.

\bibitem{WeinbergerMR333221}
{\sc H.~F. Weinberger}, {\em Remark on the preceding paper of {S}errin}, Arch.
  Rational Mech. Anal., 43 (1971), pp.~319--320.

\bibitem{WillmsMR1289661}
{\sc N.~B. Willms, G.~M.~L. Gladwell, and D.~Siegel}, {\em Symmetry theorems
  for some overdetermined boundary value problems on ring domains}, Z. Angew.
  Math. Phys., 45 (1994), pp.~556--579.

\end{thebibliography}

\end{document}